\documentclass [12pt]{article}
\usepackage{amssymb,amsmath,amsthm}

\usepackage[cp1250]{inputenc}
\usepackage[active]{srcltx}

\newtheorem{example}{Example}
\newtheorem{lemma}{Lemma}
\newtheorem{rem}{Remark}

\newtheorem{cor}{Corollary}
\newtheorem{theorem}{Theorem}
\def\supp{{\rm supp}\,}

\long\def\comment#1{{}}

\providecommand{\keywords}[1]
{
  \small	
  \textbf{\textit{Keywords---}} #1
}
\begin{document}
\title{Positivity of Tur\'an determinants for orthogonal polynomials II}
\author{Ryszard Szwarc}
\date{}

\maketitle
\begin{abstract}
The polynomials $p_n$ orthogonal on the interval $[-1,1],$ normalized by $p_n(1)=1,$ satisfy Tur\'an's inequality if
$p_n^2(x)-p_{n-1}(x)p_{n+1}(x)\ge 0$ for $n\ge 1$ and for all $x$ in the
interval of orthogonality.  We give a general criterion for orthogonal
polynomials  to satisfy Tur\'an's inequality.  This extends essentially the results of \cite{szw}. In particular the results can be applied to many classes of orthogonal polynomials, by inspecting their recurrence relation.
\end{abstract}

\section{Introduction}
Consider a symmetric probability measure $\mu$ such that $\supp\mu= [-1,1].$  By the Gram-Schmidt orthogonalization procedure applied to the system of monomials $x^n,$ $n\ge 0, $ we obtain a sequence of orthogonal polynomials $p_n(x),$ $n\ge 0.$ Every polynomial $p_n$ is of exact degree $n.$ We may assume that its  leading coefficient is positive.
It is well known that the polynomials $p_n$ satisfy the three term recurrence relation of the form
\begin{equation}\label{recrel}
xp_n= \gamma_n p_{n+1}+\alpha_np_{n-1},\quad n\ge 0,
\end{equation} with convention $\alpha_0=p_{-1}=0.$
Due to orthogonality the polynomial $p_n$ has $n$ roots  in the open interval $(-1,1).$ Therefore $p_n(1)>0.$ Let
$$P_n(x)={p_n(x)\over p_n(1)}, \quad n\ge 0.$$
The coefficients $\gamma_n, \ \alpha_{n+1}$ are positive for  $n\ge 0.$ In case the polynomials $p_n$ are orthonormal then the sequences of the coefficients are related by $\gamma_n=\alpha_{n+1}$ and the recurrence relation simplifies to
$$xp_n=\alpha_{n+1}p_{n+1}+\alpha_np_{n-1}, \qquad n\ge 0.
$$
 We refer to \cite{ch,szego} for the basic theory concerning orthogonal polynomials.

We are interested in determining when 
\begin{equation}\label{main}
\Delta_n(x):=P_n(x)^2-P_{n-1}(x)P_{n+1}(x)\ge 0, \qquad n\ge 1.
\end{equation}
The expression $\Delta_n(x)$ is called the Tur\'an's determinant. The problem has been studied for many classes of specific orthogonal polynomials (see \cite{asc, a1, bi, bi1, g1, g2, ks, mn, sk, szasz, szego1, szego2, tn, tu}. We refer to the introduction in \cite{szw} for a short account of known results.

 Tur\'an determinants can be used 
 to determine the orthogonality measure $\mu$ in terms of orthonormal polynomials  $p_n.$
Paul Nevai \cite{ne} observed
if 
  $\alpha_n\stackrel{n}{\to} 1/2$  then the sequence of measures (perhaps signed) 
$$[p_n^2(x)-p_{n-1}(x)p_{n+1}(x)]\,d\mu(x)$$ is weakly convergent to the measure 
$${2\over \pi} \sqrt{1-x^2}\,dx, \quad |x|<1.$$
M\'at\'e and Nevai \cite{mane} showed that if additionally  sequence $\alpha_n$ has bounded variation then the limit of Tur\'an determinants exists. Moreover the orthogonality measure is absolutely continuous on the interval $(-1,1)$ 
its density is given by
 $${2\sqrt{1-x^2}\over \pi f(x)}, \quad |x|<1,$$
 where 
$$f(x):=\lim_n [p_n^2(x)-p_{n-1}(x)p_{n+1}(x)]>0,\quad |x|<1.$$
 
It turns out that the way we normalize the polynomials is essential for
the Tur\'an inequality to hold.
 Indeed, assume  $p_n$ satisfy (\ref{recrel})
and $p_n(1)=1,$ i.e.
\begin{equation}\label{one}
 \alpha_n+\gamma_n=1.
\end{equation} 
Assume
$$p_n^2(x)-p_{n-1}(x)p_{n+1}(x)\ge  0,\quad | x|\le 1,\ n\ge 1.$$
Define new polynomials by $p^{(\sigma)}_n(x)=\sigma_np_n(x),$ where $\sigma_n$ is a sequence of positive coefficients.  Then the condition
$$\{p^{(\sigma)}_n(x)\}^2-p^{(\sigma)}_{n-1}(x)p^{(\sigma)}_{n+1}(x)\ge
0,\quad|x|\le 1,\ n\ge 1$$
is equivalent to (see Proposition \cite{szw})
$$\sigma_n^2 -\sigma_{n-1}\sigma_{n+1}\ge 0,\quad n\ge 1.$$
This means if the Tur\'an determinants are nonnegative, when the polynomials are normalized at $x=1,$ then they stay nonnegative for any other normalization provided that they are nonnegative at $x=1,$ as $\sigma_n=p_n^{(\sigma)}(1).$

By Theorem 1 \cite{szw} if the polynomials are normalized at $x=1,$ i.e. $p_n(1)=1,$ $\alpha_n$ is increasing and $\alpha_n\le {1\over 2},$ the Tur\'an determinants are positive in the interval $(-1,1).$ This result can be applied to many classes of orthogonal polynomials, including for example the ultraspherical polynomials for which positivity has been obtained in 
\cite{sk, szasz} 

The result mentioned above can be applied provided that we are given the coefficients $\alpha_n$ explicitly. 
For many classes of orthogonal polynomials in the interval $[-1,1]$ we are given recurrence relations, but the values $p_n(1)$ cannot be evaluated in the explicit form. Therefore we are unable to provide a recurrence relation  for the polynomials $P_n(x)=p_n(x)/p_n(1),$ in the form for which we can inspect easily the assumptions of Theorem 1 \cite{szw}. This occurs when
we study the associated polynomials. Indeed assume $p_n$ satisfy (\ref{recrel}) and (\ref{one}). For a fixed natural number the associated polynomials $p_n^{(k)}$ of order $k$ are defined
by
\begin{equation}
xp_n^{(k)}=
\begin{cases}
\gamma_kp_1^{(k)} & n=0,\\
\gamma_{n+k}p_{n+1}^{(k)}+\alpha_{n+k}p_{n-1}^{(k)} & n\ge 1.
\end{cases}
\end{equation}
These polynomials do not satisfy $p_n^{(k)}(1)=1$ as
$$p_1^{(k)}(1)=\gamma_k^{-1}=(1-\alpha_k)^{-1}>1.$$

The obstacle described above has been partially overcome in Corollary 1 of \cite{szw}, but it required additional assumptions, in particular $\gamma_0\ge 1.$ 
Unfortunately many examples including the associated polynomials violate that condition. 
The aim of this note is to provide a counterpart to Corollory 1 \cite{szw}
by allowing $\gamma_0<1.$ This is done in Theorem 1. As the assumptions in this theorem are complicated Corollary 1 provides a wide class of relatively simple recurrence relations for which Theorem 1 applies. General examples are provided at the end of the paper.
\section{Results}
\begin{theorem}
Assume the polynomials $p_n$ satisfy
\begin{equation}\label{1}
xp_n=\gamma_np_{n+1}+\alpha_np_{n-1},\quad n\ge 0,
\end{equation} where $\alpha_0=p_{-1}=0,\ p_0=1.$ 
Assume 
\begin{enumerate}
\item[(a)] the sequence $\alpha_{n}$ is strictly increasing and $\alpha_{n}\le {1/2},$
\item[(b)] the sequence $\gamma_n$ is positive and strictly decreasing ,
\item[(c)] $\alpha_n+\gamma_n\le 1.$
\end{enumerate}
Assume also that  there holds
\begin{eqnarray}
{\alpha_n-\alpha_{n-1}\over \alpha_{n}\gamma_{n-1}-\alpha_{n-1}\gamma_{n}} &\le & 
{\alpha_{n+1}\gamma_n-\alpha_n\gamma_{n+1}\over \gamma_n-\gamma_{n+1}}, \quad n\ge 1,\label{2}\\
\gamma_0-\gamma_1&\le & \alpha_1\gamma_0^2.\label{3}
\end{eqnarray}
 Then for $${P}_n(x)={p_n(x)\over p_n(1)}$$ we have
 $${P}_n(x)^2-{P}_{n-1}(x){P}_{n+1}(x)\ge 0,\qquad -1\le x\le 1.$$
\end{theorem}
\begin{proof}
Let $$g_n={p_{n+1}(1)\over p_n(1)}.$$
By (\ref{1}) we get
\begin{equation}\label{5}
g_n={1\over \gamma_n}\left (1-{\alpha_n\over g_{n-1}}\right ), \quad n\ge 1.
\end{equation}
\begin{lemma} Under assumptions of Theorem 1 there holds
\begin{equation}\label{6}
1\le g_n\le {\alpha_{n+1}\gamma_n-\alpha_n\gamma_{n+1}\over \gamma_n-\gamma_{n+1}}, \quad n\ge 0.
\end{equation}
\end{lemma}
\begin{proof}
(\ref{1}) gives $g_0=1/\gamma_0\ge 1.$ Assume $g_{n-1}\ge 1$ for $n\ge 1.$ By (\ref{5}) and (c) we get
$$g_n\ge {1\over \gamma_n}\left (1-\alpha_n\right )\ge 1.$$ This shows the left hand side inequality. 

By (\ref{3}) we get
$$g_0={1\over \gamma_0}\le {\alpha_1\gamma_0\over \gamma_0-\gamma_1},$$
which shows the right hand side inequality in (\ref{6}) for $n=0.$ Assume (\ref{6}) holds for some $n\ge 0.$ Then, in view of (\ref{5}) and (\ref{2}), we get
\begin{multline*}
g_{n+1}={1\over \gamma_{n+1}}\left (1-{\alpha_{n+1}\over g_n}\right )\le 
{1\over \gamma_{n+1}}\left ( 1-{\alpha_{n+1}(\gamma_n-\gamma_{n+1})\over \alpha_{n+1}\gamma_n-\alpha_n\gamma_{n+1}}\right )\\
= {\alpha_{n+1}-\alpha_n\over \alpha_{n+1}\gamma_n-\alpha_n\gamma_{n+1}}\le 
{\alpha_{n+2}\gamma_{n+1}-\alpha_{n+1}\gamma_{n+2}\over \gamma_{n+1}-\gamma_{n+2}}.
\end{multline*}
\end{proof}
\begin{lemma}
Under the assumptions of Theorem 1 the sequence $g_n=p_{n+1}(1)/p_n(1)$ is nonincreasing.
\end{lemma}
\begin{proof}
Let
$$f_k(x)={1\over \gamma_k}\left (1-{\alpha_k\over x}\right ),\quad x\ge 1.$$ The functions $f_k$ are nondecreasing. Moreover by a straightforward computation we get
\begin{equation}\label{f}
f_{k+1}(x)\le f_k(x),\qquad 1\le x\le {\alpha_{k+1}\gamma_k-\alpha_k\gamma_{k+1}\over \gamma_k-\gamma_{k+1}}.
\end{equation}
We have $$g_0={1\over \gamma_0},\quad g_1={1\over \gamma_1}(1-\alpha_1\gamma_0).$$ By (\ref{3}) we get $g_0\ge g_1.$ Assume $g_{n-1}\ge g_n.$
Then in view of (\ref{5}) and Lemma 1 we obtain
$$g_{n+1}=f_{n+1}(g_n)\le f_n(g_n)\le f_n(g_{n-1})=g_n.$$
\end{proof}
The polynomials ${P}_n$ satisfy
$$x{P}_n=\widetilde{\gamma}_n{P}_{n+1}+
\widetilde{\alpha}_n{P}_{n-1},\quad n\ge 0,$$
where
$$\widetilde{\alpha}_n=\alpha_n {p_{n-1}(1)\over p_n(1)},\qquad \widetilde{\gamma}_n=\gamma_n {p_{n+1}(1)\over p_n(1)}.$$
Since ${P}_n(1)=1$ we get
$$ \widetilde{\alpha}_n+\widetilde{\gamma}_n=1.$$ Moreover by Lemma 1,  Lemma 2 and (a)  the sequence 
$\widetilde{\alpha}_n$ is nondecreasing and 
$\widetilde{\alpha}_n\le 1/2.$ Thus the conclusion follows from Theorem 1(i) of \cite{szw}.
\end{proof}
\begin{rem} {\rm
As a side effect of Theorem 1 we get that the polynomials $p_n$ admit nonnegative linearization as the polynomials ${P}_n$ satisfy the assumptions of Theorem 1 in \cite{szw0}. We refer to \cite{szw1} where this  problem is discussed in detail.}
\end{rem}

The assumption (6) in Theorem 1 can be troublesome for verification in examples. However there is a wide class of examples for which (6) simplifies substantially. 
\begin{cor}
Let the polynomials $p_n$ satify (\ref{1}) with
$$\alpha_n={1\over 2}-\alpha \delta_n, \qquad \gamma_n={1\over 2} +\gamma \delta_n,\quad n\ge 0.$$ where
$\alpha\ge \gamma> 0$ and $\delta_n\searrow 0.$ 
Then the conclusion of Theorem 1 holds. 
\end{cor}
\begin{proof}
We have
$$ \alpha_{n+1}\gamma_n-\alpha_n\gamma_{n+1}={1\over 2}(\alpha+\gamma)(\delta_n-\delta_{n+1}), \quad n\ge 0.$$ Thus (\ref{2}) takes the form
$${2\alpha\over \alpha+\gamma}\le {\alpha+\gamma\over 2\gamma},$$ which is true for any  numbers $\alpha, \gamma >0.$

Next, since  $$0=\alpha_0={1\over 2}-\alpha\delta_0$$ we get $\alpha\delta_0= 1/2.$ Thus 
$$
\alpha_1\gamma_0^2=\left ({1\over 2}-{1\over 2} {\delta_1\over \delta_0}\right )\left ({1\over 2} +\gamma\delta_0\right )^2\ge {\delta_0-\delta_1\over 2\delta_0} \,2\gamma\delta_0=\gamma_0-\gamma_1.
$$
Therefore all the assumptions of Theorem 1 are satisfied.
\end{proof}
\begin{example}{\rm 
Consider the symmetric Pollaczek polynomials
$P_n^\lambda(x;a).$ They are orthogonal in the interval $[-1,1]$
and satisfy the recurrence relation
$$
xP_n^\lambda(x;a)= {n+1\over 2(n+\lambda+a)}
P_{n+1}^\lambda(x;a) + {n+2\lambda -1\over 2(n+\lambda
+a)}P_{n-1}^\lambda(x;a),$$
where the parameters satisfy $a>0,\ \lambda >0.$
Set
$$p_n(x)= {n!\over (2\lambda)_n}P_n^\lambda(x;a),$$
where $(\mu)_n=\mu(\mu+1)\ldots (\mu +n-1).$
Then the polynomials $p_n$ satisfy the recurrence relation
$$xp_n=
{n+2\lambda\over 2(n+\lambda+a)}p_{n+1}+
{n\over 2(n+\lambda+a)}p_{n-1}.$$
Observe that the  assumptions of
Corollary 1(i) of \cite{szw} are satisfied for $a\ge \lambda.$ 

\begin{rem}\rm  There is a misprint  in the formulation of Corollary 1 in \cite{szw}. The assumptions there required that
$$\lim_n\alpha_n ={1\over 2}\widetilde{a},\qquad \lim_n\gamma_n={1\over 2}\widetilde{a}^{-1}$$ with $0<\widetilde{a}<1.$ But the conclusion holds also for $\widetilde{a}=1$ with the same proof as in \cite{szw}. For symmetric Pollaczek polynomials we actually have $\widetilde{a}=1.$
\end{rem}
However for $\lambda>a$ the assumptions  of Corollary 1(ii) \cite{szw} are not satisfied  as was wrongly stated in \cite{szw}, because $\gamma_0<1.$  Instead we can apply  Corollary 1, on the previous page, with
$$\alpha=\lambda+a,\quad \gamma =\lambda-a,\quad \delta_n={1\over 2(n+\lambda+a)}.$$}
\end{example}
\begin{rem}{\rm Corollary 1 requires $\alpha\delta_0={1\over 2},$ i.e. the quantity $\delta_0$ is determined by $\alpha,$ which limits the range of examples. We will  get rid of that  assumption in the next corollary, allowing some flexibility for the quantity $\delta_0$. }
\end{rem}
\begin{cor}
Let the polynomials $p_n$ satify (\ref{1}) with
$$\alpha_0=0, \qquad \gamma_0={1\over 2}+\gamma\delta_0,$$
$$\alpha_n={1\over 2}-\alpha \delta_n, \qquad \gamma_n={1\over 2} +\gamma \delta_n,\quad n\ge 1,$$ where
$\alpha\ge \gamma> 0$ and $\delta_n\searrow 0.$ Assume also that
\begin{equation}\label{strange}{3\gamma-\alpha\over 2\gamma(\alpha+\gamma)}\le\delta_0\le  {1\over 2\alpha}.
\end{equation}
Then the conclusion of Corollary 1 holds. 
\end{cor}

\begin{rem}{\rm 
The condition $\delta_0\le 1/(2\alpha)$ is not artificial. Instead of setting $\alpha_0=0$ we could define
$$\alpha_0={1\over 2}-\alpha\delta_0.$$ The aformentioned assumption amounts to the condition $\alpha_0\ge 0.$

Observe also that the possible range for the quantity $\delta_0$ described in (\ref{strange}) is nonempty as
we always have
$${3\gamma-\alpha\over 2\gamma(\alpha+\gamma)}\le  {1\over 2\alpha}.$$}
\end{rem}
\begin{proof}
We are forced to modify the proof of the preceding corollary at  places where $\delta_0$ shows up, as $\alpha_0=0$ is no longer equal ${1\over 2}-\alpha\delta_0.$ Thus we have to make calculations concerning   (\ref{2}), for $n=1,$  and (\ref{3}), by hand.  Since $\alpha\delta_0\le {1\over 2}$ we get
$$\alpha_1\gamma_0^2\ge \left ({1\over 2}-{1\over 2} {\delta_1\over \delta_0}\right )\left ({1\over 2} +\gamma\delta_0\right )^2\ge {\delta_0-\delta_1\over 2\delta_0} \,2\gamma\delta_0=\gamma_0-\gamma_1.
$$ This gives (\ref{3}).
Next  we verify (\ref{2}) for $n=1,$ as the value $\delta_0$ is involved there on the left hand side. The inequality (\ref{2}) in this case reduces to
$${1\over \gamma_0}={2\over 1+2\gamma\delta_0}\le {\alpha+\gamma\over 2\gamma}.$$
This inequality is equivalent to the left hand side of (\ref{strange}).
\end{proof}
\begin{rem}\label{long}
\rm Corollary 1 requires that the sequence
\begin{equation}\label{7}
{\gamma_n-{1\over 2}\over {1\over 2}-\alpha_n}, \quad n\ge 1
\end{equation} is constant.
It is possible to  extend Corollary 1 to the case when the sequence in (\ref{7}) is  nondecreasing. Indeed
\begin{multline}\label{last}
\alpha_{n+1}\gamma_n-\alpha_n\gamma_{n+1} =\left [\left (\gamma_{n+1}-\textstyle{1\over 2}\right )\left (\textstyle{1\over 2}-\alpha_n\right )-\left (\gamma_{n}-\textstyle{1\over 2}\right )\left (\textstyle{1\over 2}-\alpha_{n+1}\right )\right ]
\\ +\textstyle{1\over 2}(\alpha_{n+1}-\alpha_n +\gamma_n-\gamma_{n+1}) \ge \textstyle {1\over 2}(\alpha_{n+1}-\alpha_n +\gamma_n-\gamma_{n+1}).
\end{multline}
Denote
\begin{equation}\label{uv}
u_n=\alpha_{n+1}-\alpha_n,\quad v_n=\gamma_n-\gamma_{n+1}.
\end{equation} By (\ref{last}) the assumption (\ref{2}) will be satisfied if
\begin{equation}\label{8}
(u_{n-1}+v_{n-1})(u_n+v_n)\ge 4u_{n-1}v_n.
\end{equation}
Let 
\begin{equation}\label{lambda}
v_k=\lambda_ku_k, \quad 0<\lambda_k\le 1.
\end{equation} Then (\ref{8}) takes the form
$$
(1+\lambda_{n-1})(1+\lambda_n)\ge 4\lambda_n
$$
i.e.
\begin{equation}\label{9}
\lambda_n\le {1+\lambda_{n-1}\over 3-\lambda_{n-1}}.
\end{equation}
Let
$$f(x)={1+x\over 3-x},\quad 0\le x\le 1.$$ The condition (\ref{9}) amounts to 
\begin{equation}\label{f}
\lambda_n\le f(\lambda_{n-1}).
\end{equation}
Thus (\ref{f}) implies (\ref{2}), provided that
the sequence in (\ref{7}) is nondecreasing.
As $f(x)\ge {1\over 3},$ the inequality
 (\ref{f}), and consequently (\ref{2}),  is satisfied whenever $\lambda_n\le 1/3.$
 Observe that for $y\ge 1$ we have
\begin{equation}\label{10}
 f\left ({y-1\over y+1}\right )={y\over y+2}.
\end{equation}
\end{rem}
Remark \ref{long} gives rise to new examples.
\begin{example}\rm  For $\varepsilon_n\searrow 0,$ $\delta_n\searrow \delta\ge 0,$ let
$$\alpha_n={1\over 2}-3\varepsilon_n(1+\delta_n),\qquad \gamma_n={1\over 2}+{ \varepsilon_n}, \qquad n\ge 0.$$
Then
$${\gamma_n-{1\over 2}\over {1\over 2}-\alpha_n}={1\over 3(1+\delta_n)}\nearrow {1\over 3(1+\delta)}$$ and (see (\ref{uv}) and (\ref{lambda}))
$$\lambda_n={\varepsilon_n-\varepsilon_{n+1}\over {3(\varepsilon_n-\varepsilon_{n+1})+3(\varepsilon_n\delta_n-\varepsilon_{n+1}\delta_{n+1})}}\le {1\over 3}.$$
Next
$$1+\delta_1\le 1+\delta_0={1\over 6\varepsilon_0}$$ (the last equality follows from $\alpha_0=0$).
Then
\begin{multline*}
\alpha_1\gamma_0^2=\left [{1\over 2}-3\varepsilon_1(1+\delta_1)\right ]\,\left ({1\over 2}+\varepsilon_0\right )^2 \\ \ge {1\over 2}\left ( 1-{\varepsilon_1\over \varepsilon_0}\right )\left ({1\over 2}+\varepsilon_0\right )^2
\ge {1\over 2}\left (1-{\varepsilon_1\over \varepsilon_0} \right )\, 2\varepsilon_0=\gamma_0-\gamma_1.
\end{multline*}
This gives (\ref{3}).
\end{example}

\begin{example}\label{exm}
\rm  For $a>0$  let
$$\alpha_n={1\over 2}-{a\over 2(n+a)},\qquad
\gamma_n={1\over 2} +{a\over 2(n+a+1)}.$$ Then the sequence in (\ref{7}) is increasing. Furthermore (cf. (\ref{uv}) and (\ref{lambda}))
$$ u_n={a\over 2(n+a)(n+a+1)},\quad v_n={a\over 2(n+a+1)(n+a+2)},\quad \lambda_n= {n+a\over n+a+2}.$$
By (\ref{10}) we have
$f(\lambda_{n-1}) =\lambda_n.$ Thus   (\ref{2}) is satisfied.   Next
\begin{align*}
&\gamma_0-\gamma_1=v_0={a\over 2(a+1)(a+2)}\le {a\over 2(a+1)^2}, \\ &\alpha_1\gamma_0^2={(2a+1)^2\over 8(a+1)^3}.
\end{align*}
As
$$
(2a+1)^2\ge 4(a+1)a,
$$ we get
$$\alpha_1\gamma_0^2\ge {a\over 2(a+1)^2}\ge \gamma_0-\gamma_1,$$
so
the condition (\ref{3}) is also satisfied.
\end{example}

\begin{rem}\label{rm}\rm
\ Let $$\lambda_n={y_n-1\over y_n+1}.$$  Then
\begin{equation}\label{y}
y_n={1+\lambda_n\over 1-\lambda_n}.
\end{equation}
 Moreover condition (\ref{9}) is equivalent to
\begin{equation}\label{in}
y_n\le y_{n-1}+1.
\end{equation}
\end{rem}

Using Remark \ref{rm} we can still generalize Example \ref{exm}.
\begin{example}\rm  For $a >0, \ b\ge 0$  let
$$\alpha_n={1\over 2}-{a\over 2(n+a)},\qquad
\gamma_n={1\over 2} +{a\over 2(n+a+b+1)}.$$  The sequence in (\ref{7}) is increasing. Next
\begin{align*}
&u_n={a\over 2(n+a)(n+a+1)},\\
& v_n={a\over 2(n+a+b+1)(n+a+b+2)},\\ &\lambda_n= {(n+a)(n+a+1)\over (n+a+b+1)(n+a+b+2)}.
\end{align*} 
By (\ref{y}) we get
$$y_n={n\over b+1}+{2a+b+2\over 2(b+1)}+{b^2+2b\over 2(b+1)(2n+2a+b+2)}.$$
Since $b\ge 0,$ the inequality (\ref{in}) holds. Next
\begin{align*}
&\gamma_0-\gamma_1=v_0= {a\over 2(a+b+1)(a+b+2)}\le {a\over 2(a+b+1)^2},\\
&\alpha_1\gamma_0^2={(2a+b+1)^2\over 8(a+1)(a+b+1)^2}\ge {(2a+1)^2\over 8(a+1)(a+b+1)^2} \ge {a\over 2(a+b+1)^2} .
\end{align*}
Thus
 (\ref{3}) is fulfilled.
\end{example}

{\bf Acknowledgment} I am grateful to the referees. Their remarks improved the exposition substantially.

\obeylines{
Institute of Mathematics
Wroc{\l}aw University
pl. Grunwaldzki 2/4
50--384 Wroc{\l}aw, Poland
{\tt ryszard.szwarc@uwr.edu.pl}}
\end{document}